\documentclass[a4paper]{amsart}
\usepackage{amsmath,amssymb,amsthm}
\usepackage{graphicx} 
\usepackage{color}
\usepackage{comment}

\usepackage{mathtools}
\mathtoolsset{showonlyrefs,showmanualtags}

\def\br{\mbox{\boldmath $r$}}

\def\bH{\underline{H}}
\def\bK{\underline{K}}
\def\bM{\mbox{\boldmath $M$}}
\def\uM{\underline{M}}
\def\bN{\mbox{\boldmath $N$}}
\def\R{\mathbb{R}}
\def\bR{\mbox{\boldmath $R$}}

\def\bX{\mbox{\boldmath $X$}}
\def\bY{\mbox{\boldmath $Y$}}

\def\1{\mbox{\boldmath $1$}}
\def\0{\mbox{\boldmath $0$}}

\def\<{\langle}
\def\>{\rangle}

\newcommand{\first}{\operatorname{I}}
\newcommand{\second}{\operatorname{II}}
\newcommand{\third}{\operatorname{III}}

\theoremstyle{definition}
\newtheorem{thm}{Theorem}[section]

\newtheorem{cor}[thm]{Corollary}
\newtheorem{rmk}[thm]{Remark}
\newtheorem{ex}[thm]{Example}

\title{On the governing equations of membrane O surfaces}
\author{Yoshiki Jikumaru}
\address{Faculty of Information Networking for Innovation And Design, Toyo University, 1-7-11 Akabanedai, Kita-ku, Tokyo, 115-8650, Japan}
\email{jikumaru@toyo.jp}

\begin{document}
\maketitle

\begin{abstract}
It is known that a shell membrane in equilibrium where a constant purely normal load $q_n$ acts on the membrane, and where the principal curvature lines coincide with the principal stress lines, forms an integrable system called a membrane O surface.
This paper formulates the governing equations for membrane O surfaces of the 1st and 2nd kind, which are analogues to Guichard surfaces of the 1st and 2nd kind introduced by Calapso.
Furthermore, under this formulation, we show that membrane O surfaces are a subclass of Demoulin's $\Omega$ surfaces, and that the B\"acklund transformation for membrane O surfaces preserves membrane O surfaces of the 1st and 2nd kind, respectively.
\end{abstract}

\section{Introduction}

In shell membrane theory, the middle surface of a (shell) membrane is modeled as a smooth surface in the $3$-dimensional Euclidean space $\R^3$.
The coupled system of equations, consisting of the Gauss-Mainardi-Codazzi equation for the surface and the membrane equilibrium equation for a given load, imposes constraints on the membrane geometry by adding the requirement that the principal curvature lines coincide with the principal stress lines.
In particular, in a constant purely normal load case, remarkably, the system becomes `integrable' and called a membrane O surface~\cite{RogersSchief}.
Such a condition is important when considering shapes that maximize material and mechanical properties, for example, in membrane structures within architectural design.
A representative example is the surface of constant mean curvature (CMC), which provides a mathematical model for soap bubbles. 
Several other subclasses of membrane O surfaces are also being studied~\cite{RSSLame}, \cite{RSSEnneper}.
On the other hand, from the perspective of the field known as integrable geometry, the study of surfaces with `nice' properties reduces to the study of the corresponding `nice' governing equations (soliton equations).
For example, surfaces of constant mean curvature correspond one-to-one with solutions to the sinh-Gordon equation (in special coordinates).
In this paper, we propose candidates for such `nice' governing equations for membrane O surfaces.
These are introduced as analogies to the representations of Guichard surfaces of the 1st and 2nd kind, as formulated by Calapso \cite{Calapso}.
Their B\"acklund transformations are observed by Eisenhart \cite{Eisenhart_Omega2}.
In \S \ref{sec:Preliminaries}, we briefly review the fundamental concepts of surface theory used in this paper.
In \S \ref{sec:equilibrium}, we review the equilibrium equations for a shell membrane, and reformulate them in the context of O surfaces as in \cite{RogersSchief}.
In \S \ref{sec:1st_2nd_kind}, we introduce membrane O surfaces of the 1st and 2nd kind based on geometric constraints.
Based on these formulations, we derive the governing equations and present typical concrete examples.
Furthermore, we show that the condition for Demoulin's $\Omega$ surfaces naturally follows from the same expression.
Finally, in \S \ref{sec:Backlund}, we show that the B\"acklund transformation of a membrane O surface of the 1st and 2nd kind yields a membrane O surface of the same kind.
This suggests that the equation we present is `solitonic'.

\section{Preliminaries}
\label{sec:Preliminaries}

Let us consider a surface $\Sigma$ in $\R^3$ that is locally parameterized by a vector-valued function $\br = \br(x, y)$.
Unless otherwise specified, vectors in $\R^3$ are column vectors.
The coordinate lines $x=const.$, $y = const.$ on the surface form a \textit{conjugate net} when they satisfy the hyperbolic equation of the form
\begin{equation}
\br_{xy} = (\log A_1)_y \br_x + (\log A_2)_x \br_y.
\end{equation}
We decompose the tangent vectors $\br_x$ and $\br_y$ of the surface as follows:
\begin{equation}
\br_x = A_1 \bX, \quad
\br_y = A_2 \bY.
\end{equation}
Then, the compatibility condition $(\br_x)_y = (\br_y)_x$ yields the relation
\begin{equation}
\bX_y = q \bY, \quad
\bY_x = p \bX.
\end{equation}
Here, the coefficients $p$ and $q$ are given as follows:
\begin{equation}
\label{eq:A1A2}
(A_1)_y = p A_2, \quad
(A_2)_x = q A_1.
\end{equation}
If there exists another solution $(\overline{A}_1, \overline{A}_2)$ to the equation \eqref{eq:A1A2}, then the linear system
\begin{equation}
\overline{\br}_x = \overline{A}_1 \bX, \quad
\overline{\br}_y = \overline{A}_2 \bY,
\end{equation}
naturally satisfies the compatibility, and we have another surface $\overline{\Sigma}$.
In this case, the corresponding coordinate lines on surfaces $\Sigma$ and $\overline{\Sigma}$ are parallel.
The surface $\overline{\Sigma}$ is called a \textit{Combescure transformation} of $\Sigma$.

Furthermore, the coordinates $(x, y)$ determined by imposing the orthogonality condition $\bX \cdot \bY = 0$ on coordinate lines with respect to the standard inner product of $\R^3$ are called \textit{curvature line coordinates}.
Curvature line coordinates are uniquely determined except for stretching of the coordinate lines and are preserved by the Combescure transformation.
In the curvature line coordinates, we have
\begin{equation}
(\bX \cdot \bX)_y = 2q (\bX \cdot \bY) = 0, \quad
(\bY \cdot \bY)_x = 2p (\bX \cdot \bY) = 0,
\end{equation}
therefore a suitable coordinate change gives $\bX \cdot \bX = \bY \cdot \bY = 1$.
In this case, the first fundamental form becomes $\first = A_1^2\, dx^2 + A_2^2 \, dy^2$.
For the unit normal vector field $\bN = \bX \times \bY$, we denote the principal curvatures $\kappa_1$ and $\kappa_2$, and introduce the following notations:
\begin{equation}
H_\circ = - \kappa_1 A_1, \quad
K_\circ = - \kappa_2 A_2.
\end{equation}
These are the coefficients of the third fundamental form $\third = H_\circ^2 \, dx^2 + K_\circ^2 \, dy^2$, and the Weingarten formula gives
\begin{equation}
\bN_x = H_\circ \bX, \quad
\bN_y = K_\circ \bY.
\end{equation}
These compatibility conditions are known as the Mainardi-Codazzi equations:
\begin{equation}
\label{eq:Mainardi_Codazzi}
(H_\circ)_y = p K_\circ, \quad
(K_\circ)_x = q H_\circ.
\end{equation}
From the definitions of $p$ and $q$, we can also express them as follows:
\begin{equation}
\label{eq:Mainardi_Codazzi_2}
p = \frac{(H_\circ)_y}{K_\circ} = \frac{(A_1)_y}{A_2}, \quad
q = \frac{(K_\circ)_x}{H_\circ} = \frac{(A_2)_x}{A_1}.
\end{equation}
Note that $\bN$ provides a representation of the unit sphere $\Sigma_\circ$ and is the Combescure transformation of $\Sigma$.
Under the above notation, the Gauss-Weingarten formula for the orthonormal frame $\Phi = (\bX, \bY, \bN)$ is expressed as follows:
\begin{equation}
\Phi_x = \Phi
\begin{pmatrix}
0 & p & H_\circ \\
- p & 0 & 0 \\
- H_\circ & 0 & 0
\end{pmatrix}, \quad
\Phi_y = \Phi
\begin{pmatrix}
0 & - q & 0 \\
q & 0 & K_\circ \\
0 & - K_\circ & 0    
\end{pmatrix}.
\end{equation}
The compatibility condition $(\Phi_x)_y = (\Phi_y)_x$ is given by the Mainardi-Codazzi equations \eqref{eq:Mainardi_Codazzi} and the following Gauss equation:
\begin{equation}
\label{eq:Gauss_eqn}
p_y + q_x + H_\circ K_\circ = 0.
\end{equation}

\section{Equilibrium equations for a shell membrane and an O surface connection}
\label{sec:equilibrium}

In the shell membrane theory, the geometry of the membrane is analyzed by the differential geometry of surfaces.
For details, see, for example, \cite{Love}, \cite{Novozhilov}.
For a constant purely normal load $q_n$ acting on the membrane per unit area, assume that the principal curvature lines and principal stress lines coincide.
Then the equilibrium equations are given by
\begin{equation}
\label{eq:membrane_equilibrium}
\begin{aligned}
(T_1)_x + (\log A_1)_x (T_1 - T_2) &= 0,\\
(T_2)_y + (\log A_2)_y (T_2 - T_1) &= 0,\\
\kappa_1 T_1 + \kappa_2 T_2 + q_n &= 0.
\end{aligned}
\end{equation}
Here, $T_1$ and $T_2$ are the in-plane normal stress resultants.
Hereafter, the equations \eqref{eq:membrane_equilibrium}${}_{1,2}$ are referred to as the in-plane equilibrium equations, and the equation \eqref{eq:membrane_equilibrium}${}_3$ is referred to as the out-of-plane equilibrium equation.
Following \cite{RogersSchief}, it is convenient to introduce the following notations:
\begin{equation}
\overline{A}_1 = T_2 A_1, \quad
\overline{A}_2 = T_1 A_2.
\end{equation}
Then, the in-plane equilibrium equation is expressed as follows:
\begin{equation}
(\overline{A}_1)_y = p \overline{A}_2, \quad
(\overline{A}_2)_x = q \overline{A}_1.
\end{equation}
Therefore, the existence of the Combescure transformation $\overline{\Sigma}$ is necessary and sufficient for the in-plane equilibrium equations to hold.
Furthermore, we introduce the following symbols:
\begin{equation}
\bH = (H_\circ, A_1, \overline{A}_1), \quad
\bK = (K_\circ, A_2, \overline{A}_2).
\end{equation}
Under these symbols, the representation of the Combescure-related triplet $\Sigma_\circ$, $\Sigma$, $\overline{\Sigma}$ is regarded as the $3\times 3$ matrix $\bR = (\bN, \br, \overline{\br})$.
Then we have the relation
\begin{equation}
\bR_x = \bX \bH, \quad
\bR_y = \bY \bK.
\end{equation}
The compatibility condition $(\bR_x)_y = (\bR_y)_x$ is given as follows:
\begin{equation}
\label{eq:compatibility_R}
\bX_y = q \bY, \quad
\bY_x = p \bX, \quad
\bH_y = p \bK, \quad
\bK_x = q \bH.
\end{equation}
In the matrix representation $\bR$, we can consider that the row vectors give three `dual' surfaces in the dual space $\R^3$.
Therefore, by considering $\bH$ and $\bK$ as tangent vectors to the `dual' surface, one can impose the orthogonality condition on the vectors $\bH$ and $\bK$ as in the curvature line coordinates.
Remarkably, the out-of-plane equilibrium equation \eqref{eq:membrane_equilibrium}${}_3$ provides an `orthogonality condition' via the constant matrix $\Lambda$:
\begin{equation}
\label{eq:orthogonality}
\bH \Lambda \bK^T = 0, \quad
\Lambda = 
\begin{pmatrix}
0 & 0 & 1\\
0 & -q_n & 0\\
1 & 0 & 0
\end{pmatrix}.
\end{equation}
In this case, the Combescure-related triplet $\bR = (\bN, \br, \overline{\br})$ forms an \textit{O surface} \cite{schief_unification_2003}.
Hereafter, we call the $3\times 3$ matrix $\bR = (\bN, \br, \overline{\br})$ as a \textit{membrane O surface}.
One of the crucial points of the O surface theory is the existence of the `first integral'.
That is, since $\Lambda$ is a constant matrix, we have
\begin{equation}
\label{eq:HK_norm}
(\bH \Lambda \bH^T)_y = (\bK \Lambda \bK^T)_y = 0.
\end{equation}
where we used the relation \eqref{eq:compatibility_R} and \eqref{eq:orthogonality}.
Therefore, there exist some functions $f(x)$ and $g(y)$ such that
\begin{equation}
\bH \Lambda \bH^T = -f(x), \quad
\bK \Lambda \bK^T = -g(y).
\end{equation}
More specifically, these are given by the following equations:
\begin{equation}
\label{eq:first_integral}
2\overline{A}_1 H_\circ - q_n A_1^2 = - f(x), \quad
2\overline{A}_2 K_\circ - q_n A_2^2 = - g(y),
\end{equation}
Substituting these into the out-of-plane equilibrium equation \eqref{eq:membrane_equilibrium}${}_3$ yields the constraint:
\begin{equation}
\label{eq:constraint}
f(x)K_\circ^2 + g(y) H_\circ^2 = q_n (H_\circ A_2 - K_\circ A_1)^2.
\end{equation}
Conversely, if we have some functions $f(x)$ and $g(y)$ satisfying the constraint \eqref{eq:constraint}, by defining the quantities $\overline{A}_1$ and $\overline{A}_2$ using the equation \eqref{eq:first_integral}, there exists a Combescure transformation $\overline{\Sigma}$ and the triplet $(\bN, \br, \overline{\br})$ satisfy the orthogonality condition \eqref{eq:orthogonality}.
Therefore, we have the following theorem:
\begin{thm}[\cite{RogersSchief}, Theorem 4.2]
Assume that a constant purely normal load $q_n$ acts on a membrane and the principal curvature lines coincide with the principal stress lines. 
Then the membrane is in equilibrium if and only if there exist functions $f(x)$ and $g(y)$ satisfying the constraint \eqref{eq:constraint} under the Gauss-Mainardi-Codazzi equation \eqref{eq:Mainardi_Codazzi}, \eqref{eq:Gauss_eqn}.
In this case, the stress resultants $T_1$, $T_2$ are determined by the equation \eqref{eq:first_integral}.
\end{thm}
\begin{rmk}
A class of surfaces called \textit{Guichard surfaces}, which provide a generalization of isothermic surfaces, is defined as follows.
For a surface $\Sigma$ and its Combescure transformation $\overline{\Sigma}$, we denote the principal curvatures $\kappa_1$, $\kappa_2$, $\overline{\kappa}_1$, $\overline{\kappa}_2$, respectively.
Then, if there exists a constant $c \neq 0$ satisfying
\begin{equation}
\frac{1}{\kappa_1\overline{\kappa}_2} + \frac{1}{\kappa_2\overline{\kappa}_1} = c,
\end{equation}
then, the surface $\Sigma$ is called a Guichard surface and $\overline{\Sigma}$ is called its associate.
In the context of O surfaces, the constraint can be characterized as the existence of functions $f(x)$ and $g(y)$ satisfying
\begin{equation}
f(x)A_2^2 + g(y) A_1^2 = c (H_\circ A_2 - K_\circ A_1)^2.
\end{equation}
Therefore, formally, the object obtained by swapping the first and third fundamental forms of a Guichard surface is a membrane O surface.
\end{rmk}

\section{Membrane O surfaces of the 1st and 2nd kind}
\label{sec:1st_2nd_kind}

In the following, we consider the case where $q_n \neq 0$.
Calapso, in his paper \cite{Calapso}, presented Guichard surfaces of the 1st and 2nd kind, and studied their governing equations (see also \cite{eisenhart_transformation}, \S 92).
In this section, we consider analogies in these membrane O surfaces.
In the integration of equation \eqref{eq:HK_norm}, by considering the change of variables preserving the curvature line coordinates, the first integrals $f(x)$ and $g(y)$ are normalized as follows:
(i) $f(x) = - g(y) = q_n$, (ii) $f(x) = g(y) = q_n$.
By virtue of the study by Calapso, we refer to the former as a \textit{membrane O surface of the 1st kind}, and the latter as the \textit{2nd kind}.
Depending on the normalization, we show the following theorem:
\begin{thm}
\label{thm:main1_1st}
For a membrane O surface of the 1st kind, that is, the coordinates $(x, y)$ where normalization
\begin{equation}
f(x) = - g(y) = q_n,
\end{equation}
is admissible, the first and third fundamental forms are represented as follows:
\begin{equation}
\begin{aligned}
\first &= (\cosh \alpha + h \sinh \alpha)^2 \, dx^2 + (\sinh \alpha + h \cosh \alpha)^2 \, dy^2,\\
\third &= e^{2\xi} (\sinh^2 \alpha \, dx^2 + \cosh^2 \alpha \, dy^2),
\end{aligned}
\end{equation}
In this case, $h$, $\alpha$, and $\xi$ satisfy the following system:
\begin{equation}
\label{eq:1st_govern}
\begin{aligned}
&h_x = (h + \coth \alpha) \xi_x, \quad 
h_y = (h + \tanh \alpha) \xi_y,\\
&\xi_{xy} = \xi_x \xi_y + (\log \sinh \alpha)_y \xi_x + (\log \cosh \alpha)_x \xi_y,\\
&(\alpha_x + \xi_x \coth \alpha)_x 
+ (\alpha_y + \xi_y \tanh \alpha)_y
+ e^{2\xi} \sinh \alpha \cosh \alpha = 0.
\end{aligned}
\end{equation}
Furthermore, using $h$, $\alpha$, and $\xi$, the stress resultants $T_1$ and $T_2$ are given by
\begin{equation}
\label{eq:1st_stress}
\begin{aligned}
T_1 
&= \frac{q_ne^{-\xi}}{2} \frac{2h \sinh \alpha + (1+h^2)\cosh \alpha}{\sinh \alpha + h \cosh \alpha},\\
T_2
&= \frac{q_ne^{-\xi}}{2} \frac{2h \cosh \alpha + (1+h^2)\sinh \alpha}{\cosh \alpha + h \sinh \alpha}.
\end{aligned}
\end{equation}
\end{thm}
\begin{thm}
\label{thm:main1_2nd}
For a membrane O surface of the 2nd kind, that is, the coordinates $(x, y)$ where normalization
\begin{equation}
f(x) = g(y) = q_n,
\end{equation}
is admissible, the first and third fundamental forms are represented as follows:
\begin{equation}
\begin{aligned}
\first &= (\cos \alpha + h \sin \alpha)^2 \, dx^2 + (\sin \alpha - h \cos \alpha)^2 \, dy^2,\\
\third &= e^{2\xi} (\sin^2 \alpha \, dx^2 + \cos^2 \alpha \, dy^2),
\end{aligned}
\end{equation}
In this case, $h$, $\alpha$, and $\xi$ satisfy the following system:
\begin{equation}
\label{eq:2nd_govern}
\begin{aligned}
&h_x = (h + \cot \alpha) \xi_x, \quad
h_y = (h - \tan \alpha) \xi_y,\\
&\xi_{xy} = \xi_x \xi_y + (\log \sin \alpha)_y \xi_x + (\log \cos \alpha)_x \xi_y,\\
&(-\alpha_x + \xi_x \cot \alpha)_x + (\alpha_y + \xi_y \tan \alpha)_y + e^{2\xi} \sin \alpha \cos \alpha = 0.
\end{aligned}
\end{equation}
Furthermore, using $h$, $\alpha$, and $\xi$, the stresses resultants $T_1$ and $T_2$ are given by
\begin{equation}
\label{eq:2nd_stress}
\begin{aligned}
T_1 
&= \frac{q_ne^{-\xi}}{2} \frac{2h \sin \alpha + (1-h^2)\cos \alpha}{\sin \alpha - h \cos \alpha},\\
T_2
&= \frac{q_ne^{-\xi}}{2} \frac{2h \cos \alpha - (1-h^2)\sin \alpha}{\cos \alpha + h \sin \alpha}.
\end{aligned}
\end{equation}
\end{thm}
\begin{rmk}
The second fundamental form is given by the following equations: 
For the 1st kind:
\begin{equation}
\second = -e^\xi (\sinh \alpha (\cosh \alpha + h \sinh \alpha) \, dx^2 + \cosh \alpha (\sinh \alpha + h \cosh \alpha)\, dy^2).
\end{equation}
For the 2nd kind:
\begin{equation}
\second = -e^\xi (\sin \alpha (\cos \alpha + h \sin \alpha) \, dx^2 + \cos \alpha (h \cos \alpha- \sin \alpha)\, dy^2).
\end{equation}
\end{rmk}
\begin{rmk}
Similar governing equations are obtained for both Guichard surfaces of the 1st and 2nd kind, but the equations corresponding to equations \eqref{eq:1st_govern}$_{2,3}$ and \eqref{eq:2nd_govern}$_{2,3}$ are of a form in which all three functions $\alpha$, $\xi$, and $h$ appear (see \cite{Calapso}, p.212, p.217, and \cite{eisenhart_transformation}, \S 92).
However, for membrane O surfaces, they depend only on $\alpha$ and $\xi$.
Therefore, in principle, solve the equation \eqref{eq:1st_govern}${}_{2,3}$ or the equation \eqref{eq:2nd_govern}${}_{2,3}$ to find $(\alpha, \xi)$ and then solve the equation \eqref{eq:1st_govern}${}_1$ or the equation \eqref{eq:2nd_govern}${}_1$ to obtain $h$.
\end{rmk}
\begin{ex}
\label{ex:cmc}
In the 1st kind case, we take $\xi = 0$ as a trivial solution of equation \eqref{eq:1st_govern}${}_2$.
Then, from the equation \eqref{eq:1st_govern}${}_1$, $h$ becomes a constant, and the equation \eqref{eq:1st_govern}${}_3$ becomes the (elliptic) sinh-Gordon equation:
\begin{equation}
\alpha_{xx} + \alpha_{yy} + \sinh \alpha \cosh \alpha = 0.
\end{equation}
By taking $h = 1$, the membrane stresses become isotropic and homogeneous:
\begin{equation}
T_1 = T_2 = q_n.
\end{equation}
Furthermore, the first and second fundamental forms are given by
\begin{equation}
\first = e^{2\alpha} (dx^2 + dy^2), \quad
\second = - e^\alpha (\sinh \alpha \, dx^2 + \cosh \alpha \, dy^2).
\end{equation}
Therefore, it corresponds to a surface of constant mean curvature $-1/2$ represented in conformal curvature line coordinates.
For future reference, we explicitly give several quantities:
\begin{equation}
\begin{aligned}
&A_1 = A_2 = e^\alpha, \quad
p = \alpha_y, \quad
q = \alpha_x,\\
&\overline{A}_1 = \overline{A}_2 = q_n e^\alpha, \quad
H_\circ = \sinh \alpha, \quad
K_\circ = \cosh \alpha.
\end{aligned}
\end{equation}
\end{ex}
\begin{ex}
In the 2nd kind case, if we take a trivial solution $\xi = 0$ of the equation \eqref{eq:2nd_govern}${}_2$, from the equation \eqref{eq:2nd_govern}${}_1$, $h$ is constant, and the equation \eqref{eq:2nd_govern}${}_3$ becomes the sine-Gordon equation:
\begin{equation}
-\alpha_{xx} + \alpha_{yy} + \sin \alpha \cos \alpha = 0.
\end{equation}
By taking $h = 0$, the membrane stresses are expressed by
\begin{equation}
T_1 = \frac{q_n}{2} \cot \alpha, \quad
T_2 = - \frac{q_n}{2} \tan \alpha.
\end{equation}
Furthermore, the first and second fundamental forms are given by
\begin{equation}
\first = \cos^2 \alpha \, dx^2 + \sin^2 \alpha \, dy^2, \quad
\second = \sin \alpha \cos \alpha (-dx^2 + dy^2).
\end{equation}
Therefore, in this case, it corresponds to a pseudospherical surface (surface of constant Gauss curvature $-1$) in the curvature line coordinates.
\end{ex}
\begin{ex}
In the 2nd kind case, if we take $\alpha = \pi/4$, then the equations \eqref{eq:2nd_govern}${}_{2,3}$ are rewritten as follows:
\begin{equation}
(e^{-\xi})_{xy} = 0, \quad
\xi_{xx} + \xi_{yy} + \frac{1}{2} e^{2\xi} = 0.
\end{equation}
Therefore, it corresponds to a solution of the Liouville equation with separation of variables, which is a class of membrane O surfaces with planar curvature lines \cite{RSSEnneper}.
\end{ex}
One of the advantages of these representations is that they are consistent with subclasses of Demoulin's $\Omega$ surfaces.
Here, a surface is called an \textit{$\Omega$ surface} if there exist functions $U = U(x)$ and $V = V(y)$ satisfying the following equation (see, for example, \cite{eisenhart_transformation}, \cite{Mason} for details):
\begin{equation}
\left( \frac{(\kappa_1)_x}{\kappa_1-\kappa_2} \frac{UA_1}{VA_2} \right)_y
+ \varepsilon^2 \left( \frac{(\kappa_2)_y}{\kappa_1-\kappa_2}\frac{VA_2}{UA_1} \right)_x = 0, \quad \varepsilon \in \{1, i\}.
\end{equation}
A direct calculation shows the following result:
\begin{cor}
A membrane O surface of the 1st kind under the aforementioned representation satisfies
\begin{equation}
\frac{(\kappa_1)_x}{\kappa_1-\kappa_2} \frac{A_1}{A_2} = -\alpha_x, \quad
\frac{(\kappa_2)_y}{\kappa_1-\kappa_2}\frac{A_2}{A_1} = \alpha_y.
\end{equation}
In particular, it satisfies the conditions for the $\Omega$ surface when $U = V = 1$ and $\varepsilon = 1$.
Similarly, for the 2nd kind case, we have
\begin{equation}
\frac{(\kappa_1)_x}{\kappa_1-\kappa_2} \frac{A_1}{A_2} = -\alpha_x, \quad
\frac{(\kappa_2)_y}{\kappa_1-\kappa_2}\frac{A_2}{A_1} = -\alpha_y.
\end{equation}
In particular, it satisfies the conditions for the $\Omega$ surface when $U = V = 1$ and $\varepsilon = i$.
\end{cor}
It is noted in the Appendix that a class of membrane O surfaces is a subclass of $\Omega$ surfaces in general.
\begin{proof}[Proof of Theorem~\ref{thm:main1_1st}]
In this case, the constraint \eqref{eq:constraint} takes the following form:
\begin{equation}
K_\circ^2 - H_\circ^2 = (H_\circ A_2 - K_\circ A_1)^2.
\end{equation}
We introduce three variables $h$, $\alpha$, and $\xi$ as follows:
\begin{equation}
\label{eq:1st_kind_parameters}
\begin{aligned}
H_\circ &= e^\xi \sinh \alpha, \quad
K_\circ = e^\xi \cosh \alpha,\\
A_1 &= \cosh \alpha + h \sinh \alpha, \quad
A_2 = \sinh \alpha + h \cosh \alpha.
\end{aligned}
\end{equation}
Then, since we have
\begin{equation}
\label{eq:1st_mc}
\begin{aligned}
\frac{(A_1)_y}{A_2} &= \alpha_y + \xi_y \tanh \alpha, \quad
\frac{(H_\circ)_y}{K_\circ} = \alpha_y + \frac{h_y}{1+h \coth \alpha},\\
\frac{(A_2)_x}{A_1} &= \alpha_x + \xi_x \coth \alpha, \quad
\frac{(K_\circ)_x}{H_\circ} = \alpha_x + \frac{h_x}{1+h \tanh \alpha},
\end{aligned}
\end{equation}
the Mainardi-Codazzi equation \eqref{eq:Mainardi_Codazzi_2} gives the condition for $h$:
\begin{equation}
h_x = (\coth \alpha + h) \xi_x, \quad
h_y = (\tanh \alpha + h) \xi_y.
\end{equation}
The compatibility condition $(h_x)_y = (h_y)_x$ yields equation \eqref{eq:1st_govern}${}_2$.
Furthermore, from equation \eqref{eq:1st_mc}, we have
\begin{equation}
p = \frac{(A_1)_y}{A_2} = \alpha_y + \xi_y \tanh \alpha, \quad
q = \frac{(A_2)_x}{A_1} = \alpha_x + \xi_x \coth \alpha.
\end{equation}
Therefore, the Gauss equation \eqref{eq:Gauss_eqn} gives the equation \eqref{eq:1st_govern}${}_3$.
Finally, in the derivation of the membrane stresses $T_1$ and $T_2$, using the first integral \eqref{eq:first_integral}, we have
\begin{equation}
\left\{
\begin{aligned}
2\overline{A}_1 H_\circ - q_n A_1^2 &= -q_n,\\
2\overline{A}_2 K_\circ - q_n A_2^2 &= q_n,
\end{aligned}
\right.
\iff
\left\{
\begin{aligned}
2T_2 A_1 H_\circ &= q_n (A_1^2 - 1),\\
2T_1 A_2 K_\circ &= q_n (A_2^2 + 1).
\end{aligned}
\right.
\end{equation}
Substitute expression \eqref{eq:1st_kind_parameters} to obtain the conclusion.
\end{proof}
The proof of Theorem \ref{thm:main1_2nd} is almost the same, so only a brief summary is provided here.
Under the normalization $f(x) = g(y) = q_n$, the constraint \eqref{eq:constraint} gives
\begin{equation}
K_\circ^2 + H_\circ^2 = (H_\circ A_2 - K_\circ A_1)^2.
\end{equation}
We introduce the three variables $h$, $\xi$, and $\alpha$ as follows:
\begin{equation}
\begin{aligned}
H_\circ &= e^\xi \sin \alpha, \quad
K_\circ = - e^\xi \cos \alpha,\\
A_1 &= \cos \alpha + h \sin \alpha, \quad
A_2 = \sin \alpha - h \cos \alpha.
\end{aligned}
\end{equation}
The first integral \eqref{eq:first_integral} is given by the following form:
\begin{equation}
\left\{
\begin{aligned}
2\overline{A}_1 H_\circ - q_n A_1^2 &= -q_n,\\
2\overline{A}_2 K_\circ - q_n A_2^2 &= -q_n,
\end{aligned}
\right.
\iff
\left\{
\begin{aligned}
2T_2 A_1 H_\circ &= q_n (A_1^2 - 1),\\
2T_1 A_2 K_\circ &= q_n (A_2^2 - 1).
\end{aligned}
\right.
\end{equation}

\section{B\"acklund transformations}
\label{sec:Backlund}

In \cite{Eisenhart_Omega2}, Eisenhart demonstrated that transformations of Guichard surfaces map Guichard surfaces of the 1st and 2nd kind to Guichard surfaces of the same kind (see also \cite{eisenhart_transformation}, \S 92).
This section presents an analogy for this result on membrane O surfaces.
In \cite{schief_unification_2003}, a Lax representation and B\"acklund transformation for general O surfaces were established.
The transformation is a very special case of the Ribaucour transformation, which is a restriction of Jonas-Eisenhart's fundamental transformation to the orthogonal coordinates.
As applications of these, the Lax representation and the B\"acklund transformation for membrane O surfaces are established in \cite{RogersSchief}.
In this section, we prove that restriction of the B\"acklund transformation to membrane O surfaces of the 1st and 2nd kind induces a transformation of solutions within each of these subclasses.
This fact suggests that the governing equations for membrane O surfaces of the 1st and 2nd kind are `solitonic'.

First, we review the general B\"acklund transformation for membrane O surfaces.
For details, see \cite{RogersSchief}, \S 6, and \cite{schief_unification_2003}, \S 5.
For a membrane O surface $\bR = (\bN, \br, \overline{\br})$, consider the system of linear equations
\begin{equation}
\label{eq:Lax_pair_M}
\begin{aligned}
\begin{pmatrix}
\bM\\
\uM^T
\end{pmatrix}_x &=
\begin{pmatrix}
O & m \bX \bH \Lambda\\
\bH^T \bX^T & O
\end{pmatrix}
\begin{pmatrix}
\bM\\
\uM^T
\end{pmatrix}, \\
\begin{pmatrix}
\bM\\
\uM^T
\end{pmatrix}_y &=
\begin{pmatrix}
O & m \bY \bK \Lambda\\
\bK^T \bY^T & O
\end{pmatrix}
\begin{pmatrix}
\bM\\
\uM^T
\end{pmatrix}.
\end{aligned}
\end{equation}
Here, $\bM, \uM^T \in \R^3$ and $m$ is any real number.
These compatibility conditions are the orthogonality conditions $\bX \cdot \bY = 0$, $\bH \Lambda \bK^T = 0$, and
\begin{equation}
\bX_y = q \bY, \quad
\bY_x = p \bX, \quad
\bH_y = p \bK, \quad
\bK_x = q \bH.
\end{equation}
Furthermore, the following constraint is admissible:
\begin{equation}
\bM \cdot \bM - m \uM \Lambda \uM^T = \textrm{const.}
\end{equation}
We take a solution $\bM$ and $\uM$ of the linear system \eqref{eq:Lax_pair_M} under the constraint
\begin{equation}
\label{eq:constraint_M}
\bM \cdot \bM = m \uM \Lambda \uM^T = 2M.
\end{equation}
Here, $M$ is defined by the above equation.
Then, the B\"acklund transformation of $\bR = (\bN, \br, \overline{\br})$ is given by
\begin{equation}
\bR' 
= \bR - \frac{\bM \uM}{M},
\end{equation}
and $\bR' = (\bN', \br', \overline{\br}')$ gives another memebrane O surface.
In this case, if we put
\begin{equation}
\label{eq:Backlund_aux_HK}
H = m \bH \Lambda \uM^T, \quad
K = m \bK \Lambda \uM^T,
\end{equation}
then the corresponding $\bH' = (H_\circ', A_1', \overline{A}_1')$ and $\bK' = (K_\circ, A_2', \overline{A}_2')$ are given by
\begin{equation}
\label{eq:Backlund_HK}
\bH' = \bH - \frac{H\uM}{M}, \quad
\bK' = \bK - \frac{K\uM}{M}.
\end{equation}
In the following, we express these in more concrete notations.
If we put $\bM = \lambda \bX + \mu \bY + \omega \bN$ and $\uM = (\omega', \varphi, \chi)$, then it follows from the equation \eqref{eq:Lax_pair_M} that we have the following linear system with a parameter $m$ (\cite{RogersSchief}, Equation (6.1)):
\begin{equation}
\label{eq:Lax_pair}
\begin{aligned}
\begin{pmatrix}
\lambda \\
\mu \\
\omega \\
\varphi \\
\chi
\end{pmatrix}_x &=
\begin{pmatrix}
0 & -p & m \overline{A}_1 - H_\circ & - m q_n A_1 & m H_\circ \\
p & 0 & 0 & 0 & 0 \\
H_\circ & 0 & 0 & 0 & 0 \\
A_1 & 0 & 0 & 0 & 0 \\
\overline{A}_1 & 0 & 0 & 0 & 0
\end{pmatrix}
\begin{pmatrix}
\lambda \\
\mu \\
\omega \\
\varphi \\
\chi
\end{pmatrix}, \\
\begin{pmatrix}
\lambda \\
\mu \\
\omega \\
\varphi \\
\chi
\end{pmatrix}_y &=
\begin{pmatrix}
0 & q & 0 & 0 & 0 \\
-q & 0 & m \overline{A}_2 - K_\circ & - m q_n A_2 & m K_\circ \\
0 & K_\circ & 0 & 0 & 0 \\
0 & A_2 & 0 & 0 & 0 \\
0 & \overline{A}_2 & 0 & 0 & 0
\end{pmatrix}
\begin{pmatrix}
\lambda \\
\mu \\
\omega \\
\varphi \\
\chi
\end{pmatrix}.
\end{aligned}
\end{equation}
Here we have used the fact that we can choose $\omega = \omega'$.
The compatibility conditions for this linear system are the Gauss-Mainardi-Codazzi equations \eqref{eq:Mainardi_Codazzi}, \eqref{eq:Gauss_eqn}, and the membrane equilibrium equations \eqref{eq:membrane_equilibrium}.
The constraint \eqref{eq:constraint_M} is given by the following equation:
\begin{equation}
\lambda^2 + \mu^2 + \omega^2 = 2m \omega \nu, \quad
\nu = \chi - \frac{q_n \varphi^2}{2\omega}.
\end{equation}
The B\"acklund transformation $\br'$ is constructed as follows:
\begin{equation}
\label{eq:Backlund}
\br' = \br - \frac{\varphi}{m\omega\nu} (\lambda \bX + \mu \bY + \omega \bN).
\end{equation}
Moreover, the equation \eqref{eq:Backlund_aux_HK} gives
\begin{equation}
\label{eq:eqn_H_K}
\begin{aligned}
H &= m \bH \Lambda \uM^T
= m \left( \nu H_\circ + \omega \overline{A}_1 - q_n \varphi A_1 + \frac{q_n \varphi^2}{2\omega} H_\circ \right),\\
K &= m \bK \Lambda \uM^T
= m \left( \nu K_\circ + \omega \overline{A}_2 - q_n \varphi A_2 + \frac{q_n \varphi^2}{2\omega} K_\circ \right).
\end{aligned}
\end{equation}
Therefore, from the equation \eqref{eq:Backlund_HK}, the coefficients $A_1'$, $A_2'$, $H_\circ'$, and $K_\circ'$ of the first and third fundamental form $\br'$ are given as follows:
\begin{equation}
\begin{aligned}
A_1' &= A_1 - \frac{H}{m \omega \nu} \varphi, \quad
A_2' = A_2 - \frac{K}{m \omega \nu} \varphi,\\
H_\circ' &= H_\circ - \frac{H}{m \nu}, \quad
K_\circ' = K_\circ - \frac{K}{m \nu}.
\end{aligned}
\end{equation}
Under the above notations, the following results are shown:
\begin{thm}
\label{thm:main2_1st}
For a membrane O surface of the 1st kind, by using solutions of the linear system \eqref{eq:Lax_pair}, we introduce variables $\xi'$, $\alpha'$ and $h'$ as follows:
\begin{equation}
e^{\xi'} = \frac{q_n}{2} \frac{\omega}{\nu} e^{-\xi} (1-t^2), \quad
e^{\alpha'} = e^{-\alpha} \frac{1-t}{1+t}, \quad
h' = t + \frac{\varphi}{\omega} e^{\xi'}.
\end{equation}
Here we put $t = h - \dfrac{\varphi}{\omega} e^\xi$.
Then, the coefficients of the first and third fundamental forms of $\br'$ are given by
\begin{equation}
\begin{aligned}
A_1' &= \cosh \alpha' + h' \sinh \alpha', \quad
A_2' = - (\sinh \alpha' + h' \cosh \alpha'),\\
H_\circ' &= e^{\xi'} \sinh \alpha', \quad
K_\circ' = - e^{\xi'} \cosh \alpha'.
\end{aligned}
\end{equation}
Therefore, $(\xi', \alpha', h')$ again satisfies the equation \eqref{eq:1st_govern}, and the B\"acklund transformation for membrane O surfaces preserves the membrane O surface of the 1st kind.
\end{thm}
\begin{thm}
\label{thm:main2_2nd}
For a membrane O surface of the 2nd kind, by using solutions of the linear system \eqref{eq:Lax_pair}, we introduce variables $\xi'$, $\alpha'$ and $h'$ as follows:
\begin{equation}
e^{\xi'} = \frac{q_n}{2} \frac{\omega}{\nu} e^{-\xi} (1+t^2), \quad
e^{i\alpha'} = e^{i\alpha} \frac{1-it}{1+it}, \quad
h' = -t + \frac{\varphi}{\omega}e^{\xi'}.
\end{equation}
Here we put $t = h - \dfrac{\varphi}{\omega} e^\xi$.
Then, the coefficients of the first and third fundamental forms of $\br'$ are given by
\begin{equation}
\begin{aligned}
A_1' &= \cos \alpha' + h' \sin \alpha', \quad
A_2' = \sin \alpha' - h' \cosh \alpha',\\
H_\circ' &= e^{\xi'} \sin \alpha', \quad
K_\circ' = - e^{\xi'} \cos \alpha'.
\end{aligned}
\end{equation}
Therefore, $(\xi', \alpha', h')$ again satisfies equation \eqref{eq:2nd_govern}, and the B\"acklund transformation for membrane O surfaces preserves the membrane O surface of the 2nd kind.
\end{thm}
\begin{proof}[Proof of Theorem~\ref{thm:main2_1st}]
The coefficients for the fundamental form in the 1st kind are restated as follows:
\begin{equation}
\begin{aligned}
A_1 &= \cosh \alpha + h \sinh \alpha, \quad
A_2 = \sinh \alpha + h \cosh \alpha,\\
H_\circ &= e^\xi \sinh \alpha, \quad
K_\circ = e^\xi \cosh \alpha,\\
\overline{A}_1 &= \frac{q_n}{2} e^{-\xi} ((1+h^2)\sinh \alpha + 2h \cosh \alpha),\\
\overline{A}_2 &= \frac{q_n}{2} e^{-\xi} ((1+h^2)\cosh \alpha + 2h \sinh \alpha).
\end{aligned}
\end{equation}
Under these conditions, setting $t = h - \dfrac{\varphi}{\omega} e^\xi$ yields the following equation:
\begin{equation}
\begin{aligned}
\omega \overline{A}_1 - q_n \varphi A_1 + \frac{q_n \varphi^2}{2\omega} H_\circ
&= \frac{q_n}{2} \omega e^{-\xi} ((1+t^2) \sinh \alpha + 2t \cosh \alpha),\\
\omega \overline{A}_2 - q_n \varphi A_2 + \frac{q_n \varphi^2}{2\omega} K_\circ
&= \frac{q_n}{2} \omega e^{-\xi} ((1+t^2) \cosh \alpha + 2t \sinh \alpha).
\end{aligned}
\end{equation}
Therefore, if we put
\begin{equation}
e^{\xi'} = \frac{q_n}{2} \frac{\omega}{\nu} e^{-\xi} (1-t^2), \quad
e^{\alpha'} = e^{-\alpha} \frac{1-t}{1+t},
\end{equation}
then we have
\begin{equation}
\begin{aligned}
\omega \overline{A}_1 - q_n \varphi A_1 + \frac{q_n \varphi^2}{2\omega} H_\circ
&= -\nu e^{\xi'} \sinh \alpha',\\
\omega \overline{A}_2 - q_n \varphi A_2 + \frac{q_n \varphi^2}{2\omega} K_\circ
&= \nu e^{\xi'} \cosh \alpha'.
\end{aligned}
\end{equation}
Using these, equation \eqref{eq:eqn_H_K} can be rewritten as follows:
\begin{equation}
H = m\nu (H_\circ - e^{\xi'} \sinh \alpha'), \quad
K = m\nu (K_\circ + e^{\xi'} \sinh \alpha'),
\end{equation}
Therefore, the coefficients $H_\circ'$ and $K_\circ'$ of the third fundamental form of $\br'$ are given by
\begin{equation}
H_\circ' 
= H_\circ - \frac{H}{m\nu} = e^{\xi'} \sinh \alpha',\quad
K_\circ' 
= K_\circ - \frac{K}{m\nu} = - e^{\xi'} \cosh \alpha'.
\end{equation}
Furthermore, $h'$ is defined by the following relation:
\begin{equation}
h' 
= t + \frac{\varphi}{\omega}e^{\xi'}
= h + \frac{\varphi}{\omega}(e^{\xi'} - e^\xi).
\end{equation}
Then, the coefficients $A_1'$ and $A_2'$ of the first fundamental form of $\br'$ are given by the following expressions:
\begin{equation}
\begin{aligned}
A_1' &= A_1 - \frac{H}{m \omega \nu} \varphi
= \cosh \alpha + t \sinh \alpha + \frac{\varphi}{\omega} e^{\xi'} \sinh \alpha'
= \cosh \alpha' + h' \sinh \alpha',\\
A_2' &= A_2 - \frac{K}{m \omega \nu} \varphi
= \sinh \alpha + t \sinh \alpha - \frac{\varphi}{\omega} e^{\xi'} \cosh \alpha'
= - (\sinh \alpha' + h' \cosh \alpha').
\end{aligned}
\end{equation}
Here, we used the relations
\begin{equation}
\begin{aligned}
\cosh \alpha + t \sinh \alpha &= \cosh \alpha_1 + t \sinh \alpha_1, \\
\sinh \alpha + t \cosh \alpha &= - (\sinh \alpha_1 + t \cosh \alpha_1).
\end{aligned}
\end{equation}
This proves the theorem.
The proof for Theorem~\ref{thm:main2_2nd} is exactly the same.
\end{proof}
\begin{ex}[Bianchi-Darboux transformation~\cite{bianchi_ricerche_1905}, \cite{cieslinski_darboux-bianchi_1997}]
Consider again constant mean curvature surfaces shown in Example \ref{ex:cmc}.
In this case, we can put $\chi = q_n \varphi$ in the equation \eqref{eq:Lax_pair}, and by putting $\sigma = \varphi - 2\omega$ and $\overline{m} = mq_n/2$, we have
\begin{equation}
\begin{aligned}
\begin{pmatrix}
\lambda \\
\mu \\
\omega \\
\varphi \\
\sigma
\end{pmatrix}_x &=
\begin{pmatrix}
0 & -\alpha_y & -\sinh \alpha & -\overline{m}e^{-\alpha} & \overline{m}e^\alpha \\
\alpha_y & 0 & 0 & 0 & 0 \\
\sinh \alpha & 0 & 0 & 0 & 0 \\
e^{\alpha} & 0 & 0 & 0 & 0 \\
e^{-\alpha} & 0 & 0 & 0 & 0
\end{pmatrix}
\begin{pmatrix}
\lambda \\
\mu \\
\omega \\
\varphi \\
\sigma
\end{pmatrix}, \\
\begin{pmatrix}
\lambda \\
\mu \\
\omega \\
\varphi \\
\sigma
\end{pmatrix}_y &=
\begin{pmatrix}
0 & \alpha_x & 0 & 0 & 0 \\
-\alpha_x & 0 & \cosh \alpha & \overline{m} e^{-\alpha} & \overline{m} e^\alpha \\
0 & \cosh \alpha & 0 & 0 & 0 \\
0 & e^\alpha & 0 & 0 & 0 \\
0 & -e^{-\alpha} & 0 & 0 & 0
\end{pmatrix}
\begin{pmatrix}
\lambda \\
\mu \\
\omega \\
\varphi \\
\sigma
\end{pmatrix}.
\end{aligned}
\end{equation}
A direct calculation shows $e^{\xi'} = h' = 1$, $e^{\alpha'} = -\frac{\varphi}{\sigma} e^{-\alpha}$.
This transformation is called the (classical) \textit{Bianchi-Darboux transformation} for isothermic surfaces.
\end{ex}

\section{Appendix: Relation with Demoulin's $\Omega$ surfaces}

Note the general relationship between $\Omega$ surfaces and membrane O surfaces.
In general, for a Combescure-related triplet $\br_1$, $\br_2$, $\br_3$, we denote the principal curvatures $\kappa_1$, $\kappa_2$, $\overline{\kappa}_1$, $\overline{\kappa}_2$, $l_1$, $l_2$, respectively.
Then the condition for the $\Omega$ surface is given by the constraint
\begin{equation}
\label{eq:appendix_Omega}
\frac{1}{\kappa_1 \overline{\kappa}_2} + \frac{1}{\overline{\kappa}_1 \kappa_2} 
= \frac{1}{l_1} + \frac{1}{l_2},
\end{equation}
see \cite{Mason} for details.
In particular, if the sum of the principal curvature radii of $\br_3$ is constant, it yields a Guichard surface, and in this sense, it provides a generalization of Guichard surfaces.
We consider these within the framework of O surfaces.
Let the coefficients of the first fundamental form of the surface $\br_i$ ($i=1,2,3$) be denoted by $H_i$ and $K_i$, and let $\br_4 = \bN$ be the Gauss map.
Then we have
\begin{equation}
\label{eq:appendix_orthogonality}
\begin{aligned}
\frac{1}{\kappa_1 \overline{\kappa}_2} + \frac{1}{\overline{\kappa}_1 \kappa_2} = \frac{1}{l_1} + \frac{1}{l_2}
\iff H_1 K_2 + H_2 K_1 + H_3 K_\circ + K_3 H_\circ = 0.
\end{aligned}
\end{equation}
Here, we used the relation
\begin{equation}
\begin{aligned}
H_\circ = - \kappa_1 H_1 = - \overline{\kappa}_1 H_2 = - l_1 H_3,\\
K_\circ = - \kappa_1 K_1 = - \overline{\kappa}_2 K_2 = - l_2 K_3.
\end{aligned}
\end{equation}
Therefore, if we put
\begin{equation}
\bH = (H_1, H_2, H_3, H_\circ), \quad
\bK = (K_1, K_2, K_3, K_\circ), \quad
\Lambda = 
\begin{pmatrix}
0 & 1 & 0 & 0 \\
1 & 0 & 0 & 0 \\
0 & 0 & 0 & 1 \\
0 & 0 & 1 & 0
\end{pmatrix},
\end{equation}
then the condition for $\Omega$ surface \eqref{eq:appendix_Omega} is represented as the orthogonality condition $\bH \Lambda \bK^T = 0$.

Now, we consider a specialization $(H_2, K_2) = - \dfrac{q_n}{2} (H_1, K_1)$, that is, the case where $\br_1$ and $\br_2$ are similar.
Then if we put $(H_1, K_1) = (A_1, A_2)$, $(H_3, K_3) = (\overline{A}_1, \overline{A}_2)$, then the condition \eqref{eq:appendix_orthogonality} is reduced to
\begin{equation}
-q_n A_1A_2 + \overline{A}_1 K_\circ + H_\circ \overline{A}_2= 0,
\end{equation}
which is equivalent to the out-of-plane equilibrium equation of a membrane \eqref{eq:membrane_equilibrium}.
Such $\Omega$ surfaces with self-duality are implicitly suggested in \cite{DLWS}, \cite{dKoenigs}.

\section*{Acknowledgments}
This study is supported by JSPS KAKENHI Grant No. JP24K16924 and No. JP25K21661.

\bibliographystyle{amsplain}
\bibliography{main}

\providecommand{\bysame}{\leavevmode\hbox to3em{\hrulefill}\thinspace}
\providecommand{\MR}{\relax\ifhmode\unskip\space\fi MR }
\providecommand{\MRhref}[2]{%
  \href{http://www.ams.org/mathscinet-getitem?mr=#1}{#2}
}
\providecommand{\href}[2]{#2}
\begin{thebibliography}{10}

\bibitem{bianchi_ricerche_1905}
L.~Bianchi, \emph{Ricerche sulle superficie isoterme e sulla deformazione delle quadriche}, Annali di Matematica Pura ed Applicata (1898-1922) \textbf{11} (1905), no.~1, 93--157.

\bibitem{DLWS}
F.~Burstall, U.~Hertirch-Jeromin, and W.~Rossman, \emph{Discrete linear weingarten surfaces}, Nagoya Math. J. \textbf{231} (2018), 55--88.

\bibitem{dKoenigs}
F.E. Burstall, J.~Cho, U.~Hertrich-Jeromin, M.~Pember, and W.~Rossman, \emph{Discrete {$\Omega$}-nets and {Guichard} nets via discrete {Koenigs} nets}, Proc. London Math. Soc. \textbf{126} (2023), 790--836.

\bibitem{Calapso}
P.~Calapso, \emph{Alcune superficie di guichard e le relative trasformazioni}, Annali di Matematica, Serie III \textbf{11} (1905), 201--251.

\bibitem{cieslinski_darboux-bianchi_1997}
J.~Cieśliński, \emph{The {Darboux}-{Bianchi} transformation for isothermic surfaces. {Classical} results versus the soliton approach}, Differential Geometry and its Applications \textbf{7} (1997), no.~1, 1--28.

\bibitem{Eisenhart_Omega2}
L.~P. Eisenhart, \emph{Transformation of surfaces {$\Omega$ II}}, Trans. Amer. Math. Soc. \textbf{17} (1916), 53--99.

\bibitem{eisenhart_transformation}
\bysame, \emph{Transformation of surfaces}, Priceton Univ. Press, 1923.

\bibitem{Love}
A.~E.~H. Love, \emph{A treatise on the mathematical theory of elasticity}, Dover Publications, 1944.

\bibitem{Novozhilov}
V.~V. Novozhilov, \emph{Thin shell theory}, Groningen: P. Noordhoff Ltd., 1965.

\bibitem{Mason}
M.~Pember, \emph{Lie applicable surfaces}, Comm. Anal. Geom. \textbf{28} (2020), no.~6, 1407--1450. \MR{4184823}

\bibitem{RogersSchief}
C.~Rogers and W.~K. Schief, \emph{On the equilibrium of shell membranes under normal loading. {{Hidden}} integrability.}, Proc. R. Soc. Lond., Ser. A, Math. Phys. Eng. Sci. \textbf{459} (2003), no.~2038, 2449--2462.

\bibitem{schief_unification_2003}
W.~K. Schief and B.~G. Konopelchenko, \emph{On the unification of classical and novel integrable surfaces. {I}: {Differential} geometry.}, Proc. R. Soc. Lond., Ser. A, Math. Phys. Eng. Sci. \textbf{459} (2003), no.~2029, 67--84.

\bibitem{RSSLame}
W.~K. Schief, A.~Szereszewski, and C.~Rogers, \emph{The {Lamé} equation in shell membrane theory}, J. Math. Phys. \textbf{48} (2007), no.~7, 073510, 23.

\bibitem{RSSEnneper}
\bysame, \emph{On shell membranes of {Enneper} type: generalized {Dupin} cyclides}, J. Phys. A, Math. Theor. \textbf{42} (2009), no.~40, 17.

\end{thebibliography}

\end{document}